\documentclass{article}
\usepackage{amsmath, amsthm, amssymb,tikz,hyperref}
\usepackage{graphicx}
\usepackage{color}
\usepackage{verbatim}
\usepackage{subfigure}

\newcommand{\cover}{\lessdot}
\newcommand{\ree}[1]{(\ref{#1})}

\newtheorem{thm}{Theorem}[section]
\newtheorem{lem}[thm]{Lemma}
\newtheorem{cor}[thm]{Corollary}

\newtheorem{prop}[thm]{Proposition}
\theoremstyle{definition}
\newtheorem{defi}[thm]{Definition}
\DeclareMathOperator{\supp}{supp}

\newcommand{\bfm}[1]{\boldsymbol{#1}}

\usepackage[latin1]{inputenc}
\usepackage{subfigure}

\begin{document}
\title{Applications of Quotient Posets}
\author{Joshua Hallam \\Department of Mathematics\\Michigan State University
\\ East Lansing, MI 48824-1027, USA,}
\maketitle
 \begin{abstract}
In this paper we consider the characteristic polynomial of not necessarily ranked posets.  We do so by allowing the rank to be an arbitrary function from the poset to the nonnegative integers.  We will prove two results showing that the characteristic polynomial of a poset has nonnegative integral roots.  Our factorization theorems will then be used to show that any interval of the Tamari lattice has a characteristic polynomial which factors  in this way.  Blass and Sagan's result about LL lattices will also be shown to be a consequence of our factorization theorems.  Finally we will use quotient posets to give unified proofs of some classic M\"obius function results.
\end{abstract}

\section{Introduction}
All the posets we will consider here will be finite and contain a minimum element which will be denoted by $\hat{0}$.  Our focus will be on the (one-variable) M\"obius function and its generating function, the characteristic polynomial.  In particular, we will give some new  theorems about when the characteristic polynomial of a poset has nonnegative integer roots.  Additionally, we will introduce a new method for proving some of the classic results  about the M\"obius function.  We begin with a review of the M\"obius function and the characteristic polynomial.

Here we will use $\mathbb{Z}$ to denote the set of  integers and $\mathbb{N}$ to be the set of nonnegative integers.   Given a poset $P$, the M\"obius function $\mu: P \rightarrow \mathbb{Z}$ is defined as the unique function on $P$ such that
$$
\sum_{y\leq x} \mu(y) = \delta_{\hat{0},x}
$$
where $\delta_{\hat{0},x}$ is the Kronecker delta function.

We say a poset, $P$, is \emph{ranked} if, for  each  $x\in P$, every saturated $\hat{0}$--$x$  chain  has the same length.  If a poset is ranked, the  \emph{rank function} $\rho:P\rightarrow \mathbb{N}$  is  given by setting $\rho(x)$ to be the length of a $\hat{0}$--$x$ chain.   

In the standard definition of the characteristic polynomial, one must have a ranked poset.   However, to enlarge the set of posets we can consider,  we will instead replace the rank function with any map $\rho:P\rightarrow \mathbb{N}$.  While $\rho$ is arbitrary here, certain conditions may be imposed on the function by the hypotheses of the various theorems considered later.  

Given a $\rho:P\rightarrow \mathbb{N}$ we will define the \emph{rank} of the poset as
$$
\rho(P) = \max_{x\in P} \rho(x).
$$
We are now in a position to define the generating function for  $\mu$. Let $P$ be a poset, the \emph{characteristic polynomial with respect to $\rho$ and $m$} is defined by
\begin{equation}\label{defChiEq}
\chi(P,t) = \sum_{x\in P} \mu(x) t^{m-\rho(x)}
\end{equation}
where $m$ is some integer with $m\geq \rho(P)$. 

 Although most of the results concerning the characteristic polynomial we present in this paper  will be true regardless  of whether the poset is ranked or not, we may from time to time need to assume that the poset is ranked.  In the case when $\rho$ is the normal rank function and $m=\rho(P)$ we will use the name \emph{classic characteristic polynomial} to distinguish from the more general definition.

Let us do an example and calculate the characteristic polynomial of an unranked poset.   We will consider the \emph{Tamari lattices}~\cite{ft:pa, ht:pa}, which will be denoted by $T_n$.  One way to define $T_n$ is as the set of parenthesizations of the word $x_1x_2 \cdots x_{n+1}$ with ordering   given by saying $\pi$ is covered by $\sigma$ if there exists subwords $A,B,$ and $C$ such that
$$
\pi = \dots ((AB)C)\dots \hspace{10 pt}\mbox{and}\hspace{10 pt}  \sigma = \dots(A(BC))\dots
$$
 Figure~\ref{t3Fig} displays the Hasse diagrams for $T_3$. 

\begin{figure}
\begin{center}
\begin{tikzpicture}[scale=1.5]
\node (0) at (1,0) {$(((x_1x_2)x_3)x_4)$};
\node (a) at (0,1.5) {$((x_1x_2)(x_3x_4))$}; 
\node (b) at (2,1) {$((x_1(x_2x_3))(x_4)$}; 
 
\node (d) at (2,2)  {$(x_1((x_2x_3)(x_4))$}; 
\node (1) at (1,3) {$(x_1(x_2(x_3x_4)))$}; 
\draw(0)--(a)--(1);
\draw(0)--(b)--(d)--(1);
\end{tikzpicture}
\caption{The Tamari Lattice $T_3$} \label{t3Fig}
\end{center}
\end{figure}
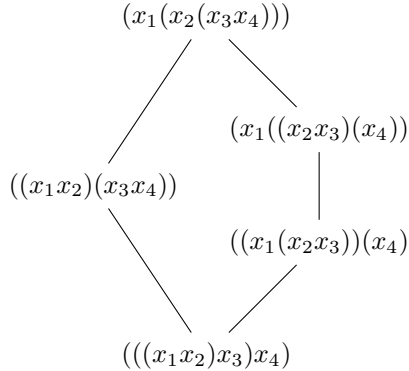

As one can see from the Hasse diagram, $T_3$ is not ranked.  In order to calculate the characteristic polynomial for $T_3$ we need a function, $\rho$.  We will use generalized rank  which was introduced in~\cite{bs:mfl}.  To define generalized rank, let us set up some notation.  The atom set of $P$ will be denoted by $A(P)$.  Additionally, given an $x\in P$ we will use $A_x$ to denote the set of atoms below $x$ in $P$.  If  $(A_1,A_2, \dots, A_n)$ is an ordered partition of the atoms of $P$, the \emph{generalized rank} of an element $x$ is given by
\begin{equation}
\label{genrho}
\rho(x) = |\{i\ :\  A_i \cap A_x\neq \emptyset\}|.
\end{equation}
In other words, $\rho(x)$ counts the number of blocks in the partition that $x$ is above.  

Returning to the $T_3$ example, let us partition the atoms into
 $$
A_1= \{((x_1x_2)(x_3x_4))\},
A_2=\{((x_1(x_2x_3))(x_4)\}).
$$
Given this partition, we see that the generalized rank of the bottom element is 0, the three middle elements all have generalized rank 1 and the top element has generalized rank 2.   We take $m=3$ which is the the length of the longest chain in $T_3$.   Using the definition of the characteristic polynomial (equation~\ree{defChiEq}), we get
$$
\chi(T_3,t) = t^3-2t^2+t= t(t-1)^2.
$$

We see that $\chi(T_3,t)$ factors with  roots 0 and 1.  Recalling the well known fact that if $P$, $Q$ are ranked then $\chi(P\times Q,t) =\chi(P,t)\chi(Q,t)$ one might ask if we can decompose $T_3$ into the product of two smaller posets. Using this reasoning, we might guess that $T_3$ is the product of two chains since chains have characteristic polynomials with roots 0 and 1.  Of course, this cannot be the case since chains are ranked and so their products are too, but $T_3$ is not ranked.  However,  it is possible to take the product of the chains, collapse elements in the Hasse diagram  without changing the characteristic polynomial and  also  get a poset isomorphic to $T_3$.

For some  posets the characteristic polynomial factors with nonnegative integer roots.  In~\cite{hs:fcp} a class of ranked lattices with this factorization was considered.  We wish to give a generalization of these results to arbitrary finite posets with a minimum element.  Many of the theorems from~\cite{hs:fcp} still hold true at this level of generality, but we will need to develop some more concepts in order to show this.

In the next section, we review the idea of homogeneous quotient posets and how they apply to the M\"obius function.     Section~\ref{tranFuncSec} contains material about transversal functions and presents the first factorization theorem.  We consider a specific type of transversal function in  Section~\ref{compTranFuncSec}.  This new type of function allows us to show another factorization theorem. We apply our factorization results to show a new family of lattices have characteristic polynomials which factor in section~\ref{tamSec}. We also prove Blass and Sagan's~\cite{bs:mfl} result about LL lattices in this section.  Section~\ref{mobFunSec} is concerned with using quotient posets to derive some classic results about the M\"obius function.  We finish with a section on future work.

\section{Quotient Posets}\label{quoPosSec}
We wish to order the classes of an equivalence relation on a poset.  We recall some definitions from~\cite{hs:fcp}.

\begin{defi} \label{quoPosetDef}
Let $P$ be a poset and let $\sim$ be an equivalence relation on $P$. We define the {\em quotient} $P/\sim$ to be the set of equivalence classes with the binary relation $\leq$ defined by $X\leq Y$ in $P/\sim$ if and only if $x\leq y$ in $P$ for some $x\in X$ and some $y\in Y$. 
\end{defi}
 Quotients of posets are not necessarily posets.   For example, take a 3-element chain and identify the bottom and top elements.   The relation you obtain is reflexive and transitive, but   not antisymmetric.   In order to guarantee we get a poset when we take a quotient, we require two more properties.

\begin{defi} \label{homogenQuoDefi}
Let $P$ be a poset and let $\sim$ be an equivalence relation on $P$. Order the equivalence classes as in the previous definition. We say the poset $P/\sim$ is a \emph{homogeneous quotient} if
\begin{itemize}
\item[(1)] $\hat{0}$ is in an equivalence class by itself, and 
\item[(2)] if $X\leq Y$ in $P/\sim$, then for all $x\in X$ there is a $y\in Y$ such that $x\leq y$.
\end{itemize}
\end{defi}

It was shown in~\cite{hs:fcp} that homogeneous quotients of finite posets are posets.  Moreover, it was also shown how the   M\"obius function behaved when taking quotients.   We describe this next.

We say that a homogeneous quotient $P/\sim$ satisfies the \emph{summation condition} if for all nonzero $X\in P/\sim$,
\begin{equation}\label{sumCondEq}
\sum_{y \in L(X)} \mu(y) = 0
\end{equation}
where $L(X)$ is the lower order ideal generated by $X$ in $P$. 
This definition leads us to our first lemma.

\begin{lem}[\cite{hs:fcp}]\label{sumLem}
Let $P/\sim$ be a homogeneous quotient poset which satisfies the summation condition.  Then, for all equivalence classes $X$
\begin{displaymath}
\mu(X) = \sum_{x\in X} \mu(x).
\end{displaymath}
\end{lem}
 
Not only will this  lemma allow us to prove  results about the factorization of the characteristic polynomial, we will also be able to use it to prove some classic results about the M\"obius function.  We will first consider the factorization theorems.   

\section{Transversal Functions}\label{tranFuncSec}
We begin this section by reviewing the notion of a rooted tree which was used in~\cite{hs:fcp}.
\begin{defi}
Let $P$ be a poset and $S$ be a subset of $P$ which contains $\hat{0}$. Let $\mathcal{C}$ be the collection of saturated chains of $P$ which start at $\hat{0}$ and use only elements of $S$. The \emph{rooted tree with respect to $S$} is the poset obtained by ordering $\mathcal{C}$ by containment and will be denoted by $RT_S$. 
\end{defi}

First, let us note that a rooted tree contains a minimum element corresponding to the $\hat{0}$--$\hat{0}$ chain.  Additionally,  there are no cycles in the Hasse diagram of a rooted tree.  These two properties motivate the name for the poset. It also implies that only the minimum element and the atoms of the rooted tree have  nonzero  M\"obius values. 

 By definition, if $x\in RT_S$, then $x$ is a chain of the original poset.  However, it will be useful to think of $x$ as just the top element of the  chain.  That is, we think of $x$ as just an element of the original poset.

Let us a do an example of constructing rooted trees.    Consider the Tamari lattice, $T_3$, that was shown previously in Figure~\ref{t3Fig}. We will take $S_1$ to be the upper order ideal generated by $((x_1x_2)(x_3x_4))$  together with $\hat{0}$ and $S_2$ to be the upper order ideal generated by $((x_1(x_2x_3))(x_4)$ together with $\hat{0}$.  This gives $RT_{S_1}$ and $RT_{S_2}$ as shown in  Figure~\ref{t3RootedTreeFig}.

\begin{figure}
\begin{center}
\begin{tikzpicture}[scale=1.5]
\node (0) at (1,0) {$(((x_1x_2)x_3)x_4)$};
\node (a) at (1,1) {$((x_1x_2)(x_3x_4))$}; 
\node (1) at (1,2) {$(x_1(x_2(x_3x_4)))$};  

\node (0') at (4,0) {$(((x_1x_2)x_3)x_4)$};
\node (b) at (4,1) {$((x_1(x_2x_3))(x_4)$}; 
\node (d) at (4,2)  {$(x_1((x_2x_3)(x_4))$}; 
\node (1') at (4,3) {$(x_1(x_2(x_3x_4)))$};

\draw(0)--(a)--(1);
\draw(0')--(b)--(d)--(1');
\end{tikzpicture}
\caption{The Rooted Trees $RT_{S_1}$ and $RT_{S_2}$} \label{t3RootedTreeFig}
\end{center}
\end{figure}
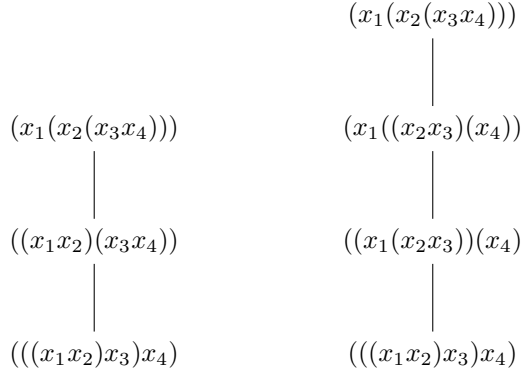

To explain factorization of the characteristic polynomial of a poset, we will take products of rooted trees and then take a quotient of this product.  We will denote elements of the product using boldface.  So if $S_1,S_2, \dots, S_n$ is a collection of subsets which contain $\hat{0}$, then a typical element of $\prod_{i=1}^n RT_{S_i}$ will be denoted $\bfm{t}=(t_1,t_2, \dots, t_n)$.

When the poset is a lattice there is a canonical choice for the equivalence relation called the \emph{standard equivalence relation} which was introduced in~\cite{hs:fcp}.  Since we are interested in posets which are not necessarily lattices we need to generalize this idea.   To do this, we quotient  out by the kernel of a  special type of map from the product of rooted trees to the poset.

\begin{defi}\label{transFuncDefi}
Let $P$ be a poset   and let  $(S_1, S_2, \dots, S_n)$ be  an ordered collection of subsets of $P$  each containing $\hat{0}$. We say $f: \prod_{i=1}^n RT_{S_i} \rightarrow P$ is a  {\it transversal function} if  it has the following properties:
\begin{enumerate}
\item The function $f$ is order preserving.
\item The function $f$ is surjective.
\item If $f(\bfm{t})=\hat{0}$, then $t_i =\hat{0}$ for all $i$.
\end{enumerate}
\end{defi}

If $f$ is a transversal function, the \emph{kernel of $f$}, denoted $\ker f$, is the equivalence relation $\sim$ given by $\bfm{s}\sim \bfm{t}$ if and only if $f(\bfm{s})=f(\bfm{t})$.  Since we will often be referring to equivalence classes and the elements of these classes we need names for these objects.

\begin{defi}
Let $P$ be a poset  and let  $(S_1, S_2, \dots, S_n)$ be  an ordered collection of subsets of $P$.   Let $f: \prod_{i=1}^n RT_{S_i} \rightarrow P$ be a transversal function.  If $\bfm{t}\in \prod_{i=1}^n RT_{S_i}$ then we say $\bfm{t}$ is a {\it transversal} for $x$ if $f(\bfm{t})=x$.  We say $\bfm{t}$ is {\it atomic} or an {\it atomic transversal} if  all the elements of $\bfm{t}$ are atoms of $RT_{S_i}$ or $\hat{0}$.   The set of all transversals for $x$ will be denoted by $\mathcal{T}_x$ and the set of all atomic transversals will be denoted by $\mathcal{T}_x^a$.  We also define the  \emph{support} of a transversal, $\bfm{t}$, as 
$$
\supp\bfm{t} = \{i :t_i\neq \hat{0}\}.
$$
\end{defi}

From the definitions it is evident that the set of equivalence classes of $ \prod_{i=1}^n RT_{S_i}/\ker f $ is 
 $\{\mathcal{T}_x : x\in P\}$.  Moreover, it is clear that the size of the  support of an atomic transversal for $x$ is also its rank  in the product of the rooted trees.   

We are now in a position to give our first factorization theorem.  The other factorization result we provide later will   be a special case of this  one.   

\begin{thm}\label{mainThm}
Let $P$ be a poset   with  $\rho: P \rightarrow \mathbb{N}$ and let $m\in \mathbb{N}$ such that $\rho(P)\leq m$. Moreover,  let $(S_1, S_2, \dots, S_n)$ be an  ordered collection of subsets of $P$ which contain $\hat{0}$ and let $f$ be a transversal function.  Suppose the following hold.
\begin{enumerate}
\item[(1)] If $x\leq y$ and $\bfm{s}\in\mathcal{T}_x$, there exists $\bfm{t}\in\mathcal{T}_y$ with $\bfm{s}\leq\bfm{t}$.
\item[(2)]  If $\bfm{t}\in \mathcal{T}_x^a$, then $|\supp \bfm{t}|=\rho(x)$.
\item[(3)] The summation condition \ree{sumCondEq} holds for all $\mathcal{T}_x$.
\end{enumerate}

We can conclude the following.
\begin{itemize}
\item[(a)] We have an isomorphism
$$
P \cong \left( \prod_{i=1}^n RT_{S_i}\right)/\ker f.
$$
\item[(b)]For each $x\in P$,
$$
\mu(x) = (-1)^{\rho(x)}|\mathcal{T}_x^a|.
$$
\item[(c)] The characteristic polynomial of $P$  with respect to $\rho$ and $m$  (equation~\ree{defChiEq}) is given by
$$
\chi(P,t)= t^{m-n}\prod_{i=1}^n (t- |A(RT_{S_i})|).
$$
\end{itemize}
\end{thm}

\begin{proof}
First, we need to show the quotient is a homogeneous quotient.  Conditions (2) and   (3)  in the definition  of a transversal function (Definition~\ref{transFuncDefi}) imply condition (1) of a homogeneous quotient (Definition~\ref{homogenQuoDefi}).  To show  condition (2) holds, suppose that $\mathcal{T}_x\leq \mathcal{T}_y$.  Then there is a $\bfm{q}\in\mathcal{T}_x$ and a $\bfm{r}\in\mathcal{T}_y$ with $\bfm{q}\leq \bfm{r}$.  Since $f$ is order preserving, $x=f(\bfm{q})\leq f(\bfm{r})=y$.  By assumption (1)  of the theorem, given a $\bfm{s}\in\mathcal{T}_x$ there is a $\bfm{t}\in\mathcal{T}_y$ with $\bfm{s}\leq \bfm{t}$ and so condition (2) of a homogeneous quotient is satisfied.

Now  we show (a).  Let $\bar{f}:\left( \prod_{i=1}^n RT_{S_i}\right)/\ker f \rightarrow P $ be the induced quotient map sending $\mathcal{T}_x$ to $x$.   Since $f$ is surjective, it follows easily that $\bar{f}$ is a bijection and so has an inverse say $g$.

Next we show $\bar{f}$ is order preserving.  Recall that the elements of the quotient, $ (\prod_{i=1}^nRT_{S_i})/\ker f$, are of the form $\mathcal{T}_x$ for some $x\in P$. Suppose that $\mathcal{T}_x \leq \mathcal{T}_y$.  Then again, since $f$ is order preserving,  $x\leq y$ and so $\bar{f}$ is order preserving.

To finish the proof of (a), we show $g$ is order preserving.  Suppose that $x\leq y$.  Since $f$ is surjective, $\mathcal{T}_x\neq\emptyset$. Therefore, by assumption (1), there   are $\bfm{s} \in \mathcal{T}_x$ and $\bfm{t}\in \mathcal{T}_y$ with $\bfm{s}\leq \bfm{t}$.  Using the definition of a quotient poset, we get that that $\mathcal{T}_x\leq\mathcal{T}_y$ and so $g(x)\leq g(y)$.

Now we verify (b).    By Lemma~\ref{sumLem}, assumption (3), and the fact that isomorphisms preserve M\"obius   values, we have that 
$$
\mu(x) = \sum_{\bfm{t} \in \mathcal{T}_x} \mu(\bfm{t}).
$$
Since  only atomic transversals have nonzero M\"obius value we have
$$
\mu(x) = \sum_{\bfm{t} \in \mathcal{T}_x^a} \mu(\bfm{t}).
$$
By assumption (2), all the atomic transversals have the same support size which is the rank of $x$.  It  follows that each  atomic transversal for $x$ has M\"obius value $(-1)^{\rho(x)}$. Therefore we have that
$$
\mu(x) =(-1)^{\rho(x)}|\mathcal{T}_x^a|.
$$

Finally we show (c).  By definition, 
$$
\chi(P,t) = \sum_{x\in P} \mu(x) t^{m-\rho(x)}.
$$
Using part (b), we get
$$
\chi(P,t) = \sum_{x\in P}  (-1)^{\rho(x)}|\mathcal{T}_x^a| t^{m-\rho(x)}.
$$
We can break this sum into parts, depending on the rank of $x$. Note that by assumption (2) and part (b), every element with rank larger than $n$ has M\"obius value zero.  Thus we have,
$$
\chi(P,t) = \sum_{k=0}^{n}\left(\sum_{\rho(x)=k}(-1)^k|\mathcal{T}_x^a|t^{m-k}\right).
$$
Neither $(-1)^k$ nor $t^{m-k}$ depend on $x$ so we can pull them out to get,
$$
\chi(P,t) = \sum_{k=0}^{n}(-1)^kt^{m-k}\left(\sum_{\rho(x)=k} |\mathcal{T}_x^a|\right).
$$
Using assumption (2) and denoting the $k^{th}$ elementary symmetric function as $e_k$, we  have the inner sum  is exactly $e_k(|A(RT_{S_1})|,|A(RT_{S_2})|,\dots, |A(RT_{S_n})|)$.
 It follows that,
$$
\chi(P,t) = \sum_{k=0}^{n}(-1)^ke_k(|A(RT_{S_1})|,|A(RT_{S_2})|,\dots, |A(RT_{S_n})|)t^{m-k}.
$$
Pulling out a factor of $t^{m-n}$ permits us to rewrite the sum as a product
$$
\chi(P,t) = t^{m-n}\prod_{i=1}^n (t- |A(RT_{S_i})|)
$$
completing the proof.
\end{proof}

\section{Complete Transversal Functions}\label{compTranFuncSec}
By definition, a transversal function must be surjective.  However, if we impose more structure on the choice of subsets used to build the rooted trees, we can remove this assumption.  In order to show this, we begin with a definition.
 
\begin{defi}
Let $P$ be a poset  and let $A$ be a set of atoms.   The {\it complete tree} (with respect to $A$) is the rooted tree $RT_{\hat{U}(A)}$ where $\hat{U}(A)$ is the upper order ideal generated by the set $A$ together with $\hat{0}$.
\end{defi}

Along with this new definition, we have a new type of function.

\begin{defi}\label{compTranFuncDefi}
Let $P$ be a poset   and let  $(A_1, A_2, \dots, A_n)$ be  an ordered partition of $A(P)$.  We say $f: \prod_{i=1}^n RT_{\hat{U}(A_i)} \rightarrow P$ is a  {\it complete transversal function} if  it is order preserving and has the property that if  in $\bfm{t}$ we have  $t_i=\hat{0}$ or $t_i=x$ for all $i$, then $f(\bfm{t})= x$.
\end{defi}

Note that it may appear that complete transversal functions are not transversal functions because we dropped the condition that they are surjective.   However, we will see in the next lemma that, among other nice properties, the surjectivity of the function is a consequence of the definition. We also note that if we have a lattice, then $f(\bfm{t})=\vee \bfm{t}$ is a complete transversal function where $\vee\bfm{t}=t_1\vee\dots\dots\vee t_n$ .

It will be useful  to have notation for a new transversal obtained by inserting an element into a preexisting transversal.  To do this, we will use $\bfm{t}(e^{i})$ to denote the transversal which is obtained by replacing the $i^{th}$ coordinate of $\bfm{t} =(t_1,t_2,\dots,t_n)$ with an element $e$. So we have,
$$
\bfm{t}(e^{i}) =(t_1, t_2, \dots, t_{i-1}, e, t_{i+1}, \dots, t_n).
$$

\begin{lem}\label{completeTreeLem}
Let $P$ be a poset   and let  $(A_1, A_2, \dots, A_n)$ be  an ordered partition of $A(P)$.   Let $f$ be a complete transversal function.  Then we can conclude the following.
\begin{itemize}
\item[(a)] The function $f$ is surjective and  $f$ is a transversal function.
\item[(b)] For all $j$, $t_j\leq f(t_1,t_2,\dots, t_n)$.
\item[(c)]  If $x\leq y$ and $\bfm{s}\in\mathcal{T}_x$, there exists $\bfm{t}\in\mathcal{T}_y$ with $\bfm{s}\leq\bfm{t}$.
\item[(d)] For $x\in P$, let  $N_i$  be the number of atoms below $x$ in $A_i$, then 
\begin{equation}\label{sumOfMob}
\sum_{\bfm{s}\in L(\mathcal{T}_x)} \mu(\bfm{s}) =  \prod_{i=1}^n(1-N_i).
\end{equation}
\item[(e)] The summation condition~\ree{sumCondEq} holds for all $\mathcal{T}_x$  if and only if for all nonzero $x\in P$, there is an index $i$ such that $|A_i\cap A_x|=1$.
\end{itemize}
\end{lem}
\begin{proof}
First we show (a).  Let $\bfm{\hat{0}}$ be the transversal having all components equal to $\hat{0}$.  Since we are using complete trees and a partition of the atom set, for every $x\in P$  there  exists some $i$ such that $\bfm{\hat{0}}(x^i)$  is a transversal.  It follows from the definition of a complete transversal function that $f(\bfm{\hat{0}}(x^i))=x$ and so $f$ is surjective.  

To show that the third condition for a transversal function holds, suppose that $f(\bfm{t})=\hat{0}$.    By definition of a complete transversal function, $f(\bfm{\hat{0}}(t_i^i)) = t_i$.  Since $f$ is order preserving and $\bfm{\hat{0}}(t_i^i)\leq \bfm{t}$ we get that $t_i = f(\bfm{\hat{0}}(t_i^i))\leq f(\bfm{t})= \hat{0}$. Therefore, if $f(\bfm{t})=\hat{0}$, then $\bfm{t}= \bfm{\hat{0}}$.  This completes the proof that $f$ is a transversal function.

For (b), we noted in the previous paragraph that
$$
f(\bfm{\hat{0}}(t^j_j)) = t_j.
$$
Using the fact that $f$ is order preserving, we get that 
$$
t_j=f(\bfm{\hat{0}}(t^j_j))\leq f(t_1,t_2,\dots, t_n).
$$

Next we prove (c).  This is trivial if $x=\hat{0}$ so assume $x$ is nonzero.   Let $\bfm{s}\in\mathcal{T}_x$.  Then by by part (b), $s_i\leq x$ for all $i$.  Let $\bfm{t}$ be given by $t_i=y$ for all $i$ with $i\in\supp \bfm{s}$ and $t_i=\hat{0}$ for all other $i$.  Such a $\bfm{t}$ is a valid transversal since $s_i\le x\le y$ and we are using complete trees. Note also that since $x\neq\hat{0}$ it must be that $\bfm{t}$ has at least one nonzero coordinate.    It follows that $\bfm{t}\in\mathcal{T}_y$ and $\bfm{s}\leq \bfm{t}$.

Next, let us show  (d).  We start by showing that 
\begin{equation}
\label{L(T_x)}
L(\mathcal{T}_x)=\{\bfm{t} \mbox{ a transversal } : t_i\leq x \mbox{ for all } i\}.
\end{equation}

To see that $L(\mathcal{T}_x)$ is contained in the other set, let $\bfm{t}\in L(\mathcal{T}_x)$.  Then for each $i$ we have $\bfm{\hat{0}}(t_i^i) \in L(\mathcal{T}_x)$.  By definition of a complete transversal function, $f(\bfm{\hat{0}}(t_i^i))=t_i$.  Since $f$ is order preserving,  $t_i=f(\bfm{\hat{0}}(t_i^i))\leq f(\bfm{t}) = x$. 

For the reverse inclusion, suppose that $\bfm{t}$ is a transversal with $t_i\leq x$  for all $i$.  Let $\bfm{s}$ be the transversal obtained from $\bfm{t}$ by replacing all the nonzero $t_i$ with $x$.  We know that  $\bfm{s}$ is a valid transversal because we are using complete trees.  Since $f$ is a complete transversal function, $f(\bfm{s})=x$ and so $\bfm{s}\in L(\mathcal{T}_x)$.  By construction, $\bfm{t}\leq \bfm{s}$ and therefore $\bfm{t}\in L(\mathcal{T}_x)$.

 Let $I$ be the set of indices, $i$, such that there is an atom below $x$ in $A_i$. By relabeling, if necessary, we may assume that  $I=\{1,2, \dots, j\}$.  Since $N_i=0$ implies that $1-N_i=1$,
$$
\prod_{i=1}^j(1-N_i)=\prod_{i=1}^n(1-N_i).
$$

From equation~(\ref{L(T_x)}) we can conclude that the number of atomic transversals in $ L(\mathcal{T}_x)$ with support size $i$ is $e_{i}(N_1, N_2, \dots, N_j)$ where $e_i$ is the $i^{th}$ elementary symmetric function.   Now for each atomic transversal $\bfm{s} \in L(\mathcal{T}_x)$ we have that $\mu(\bfm{s}) = (-1)^{|\supp(\bfm{s})|}$
and all other transversals have M\"obius value zero.  Therefore, 
$$
\sum_{\bfm{s}\in L(\mathcal{T}_x)} \mu(\bfm{s}) = \sum_{i=0}^j (-1)^i e_i(N_1, N_2, \dots, N_j)=  \prod_{i=1}^j(1-N_i)
$$
which completes part (d).

Finally, (e) follows immediately from (d) and the definition of the summation condition.
\end{proof}

Given this lemma, we can use  Theorem~\ref{mainThm} to immediately obtain the following.

\begin{thm}\label{mainThmCor2}
Let $P$ be a poset   with  $\rho: P \rightarrow \mathbb{N}$ and let $m\in \mathbb{N}$ such that $\rho(P)\leq m$. Moreover,  let $(A_1, A_2, \dots, A_n)$ be an  ordered partition of $A(P)$ and let $f$ be a complete transversal function.  Suppose the following hold.
\begin{enumerate}
\item[(1)]  If $\bfm{t}\in \mathcal{T}_x^a$, then $|\supp \bfm{t}|=\rho(x)$.
\item[(2)]  For all $x\in P$, there is an index $i$ such that $|A_i\cap A_x|=1$.
\end{enumerate}
We can conclude the following.
\begin{itemize}
\item[(a)] We have an isomorphism
$$
P \cong \left( \prod_{i=1}^n RT_{\hat{U}(A_i)}\right)/\ker f.
$$
\item[(b)]For each $x\in P$,
$$
\mu(x) = (-1)^{\rho(x)}|\mathcal{T}_x^a|.
$$
\item[(c)] The characteristic polynomial of $P$  with respect to $\rho$ and $m$  is given by
$$
\chi(P,t)= t^{m-n}\prod_{i=1}^n (t- |A_i|).
$$
\end{itemize}
\end{thm}

The reader may  be wondering why we did not just assume from the start that we were using complete transversal functions.  By doing so, we reduce the number of things we need to check and we still get the same conclusions  as in Theorem~\ref{mainThm}.  However, there are situations where the first theorem applies but the second does not.

\begin{figure}
\begin{center}
\begin{tikzpicture}
\node (0) at (0,0) {$1^0/2^0/3^0 $};
\node (a) at (-5,2) {$12^0/3^0 $};
\node (b) at (-3,2) {$13^0/2^0 $};
 \node (c) at (-1,2) {$1^0 /23^0 $};
 \node (d) at (1,2) {$12^1/3^0 $};
 \node (e) at (3,2) {$13^1/3^0 $};
 \node (f) at (5,2) {$1^0/23^1 $};
\node (g) at (-3,4) {$123^0 $};
\node (h) at (0,4) {$123^1$};
\node (i) at (3,4) {$123^2$};

\draw(0)--(a);
\draw(0)--(b);
\draw(0)--(c);
\draw(0)--(d);
\draw(0)--(e);
\draw(0)--(f);
\draw(h)--(a)--(g);
\draw(h)--(b)--(g);
\draw(h)--(c)--(g);
\draw(h)--(d)--(i);
\draw(h)--(e)--(i);
\draw(h)--(f)--(i);

\end{tikzpicture}
\caption{The Weighted Partition Poset $\Pi_3^w$} \label{wpi3Fig}
\end{center}
\end{figure}
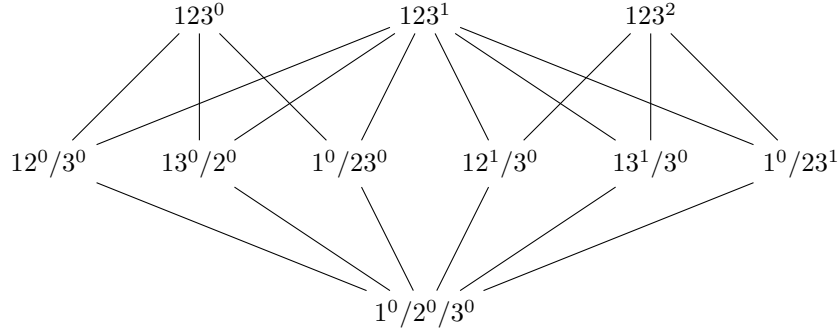

Let us give an example were the summation condition~\ree{sumCondEq} for $\mathcal{T}_x$  needed in Theorem~\ref{mainThm} holds, but the second condition of Theorem~\ref{mainThmCor2} does not.  We will consider the weighted partition poset, $\Pi_n^w$ introduced in~\cite{dk:cfopcbbho}. The elements of $\Pi_n^w$ are set partitions of $\{1,2,\dots, n\}$  where  each block $B_i$ has one of the following weights $\{0,1,\dots, |B_i|-1\}$.   The weighted partitions will be denoted by $B_1^{w_1}/B_2^{w_2}/ \dots/B_n^{w_n}$ where $w_i$ is the weight of block $B_i$. The ordering is given by 
$$
A_1^{v_1}/A_2^{v_2}/ \dots/A_k^{v_k}\leq B_1^{w_1}/B_2^{w_2}/ \dots/B_n^{w_n}
$$
if and only if 
\begin{enumerate}
\item We have
$$
A_1/A_2 /\dots/A_k \leq B_1/B_2/ \dots/B_n
$$
in the (unweighted) partition lattice $\Pi_n$.
\item If $B_l = A_{i_1}\cup A_{i_2} \cup \dots \cup A_{i_m}$, then 
$$
v_l -(w_{i_1}+w_{i_2}+\dots+ w_{i_m}) \in \{1,2,\dots, m-1\}.
$$
\end{enumerate}

The weighted partition poset $\Pi_3^w$ is shown in Figure~\ref{wpi3Fig}.  It is easy to check that the classic characteristic polynomial of this poset factors as
$$
\chi(\Pi_3^w,t) = (t-3)^2.
$$
Consider the sets 
$$
A_1 = \{12^0/3^0, 13^0/2^0, 12^1/3^0,\hat{0}\}
$$
 and 
$$
A_2= \{ 1^0 /23^0,13^1/3^0, 1^0/23^1,\hat{0}\}.
$$
Additionally, consider the transversal function $f:\prod_{i=1}^2 RT_{A_i} \rightarrow \Pi_3^w$  which sends any pair which contains $\hat{0}$ to the other element in the pair and sends any pair with two non-zero elements  to $123^i$ where $i$ is the sum of their exponents.  It is easy to check that $f$ is a transversal function and that the summation condition~\ree{sumCondEq} is satisfied.  However, the element $123^1$ is above every atom so it is impossible that it is above only one atom of either $A_1$ or $A_2$.   One can also check that all the conditions of Theorem~\ref{mainThm} are satisfied and so we have verified that the classic characteristic polynomial does factor using our method.

We should also point out that, as was shown in~\cite{dw:cpwp}, the classic  characteristic polynomial of the weighted partition poset $\Pi_n^w$   factors as $\chi(\Pi_n^w,t) = (t-n)^{n-1}$. This was shown using different methods than presented here.  As of now, we do not have a transversal function which gives us the factorization.   

\section{Tamari Lattices and LL Lattices}\label{tamSec}

Despite  the reason explained earlier, Theorem~\ref{mainThmCor2}  can be quite useful.   First, we will show how to use it to explain the factorization of any interval of the Tamari lattice which implies a factorization result for both $m$-Tamari lattices, originally defined in~\cite{bp:htdhgtp}, and the standard Tamari lattice.   We note that these results are new.   We will also use the theorem to give a nice formula for the  M\"obius function of intervals in the Tamari lattice.  Finally, we will show that the theorem implies a result of Blass and Sagan~\cite{bs:mfl} concerning   LL lattices.

To show that  the characteristic polynomial of the  intervals of the Tamari lattices factor,  we introduce a different way to denote the elements.  For each element of $\sigma\in T_n$ we will give a corresponding  \emph{left-bracket vector} $(v_1,v_2, \dots, v_n)$. The value of  $v_i$  is obtained by locating $x_i$ in $\sigma$ then, moving left,  counting the number of $x$'s (including $x_i$) and left parentheses that you  pass  until the two numbers are the same.  At this point, stop and set $v_i=j$ where $x_j$ is the last $x$ that was passed before the two numbers became equal.

Let us do an example of calculating a  left-bracket vector. Suppose that  we let $\sigma=((x_1x_2)(x_3x_4))$, then to find $v_1$ look for $x_1$ and move left.  Immediately to the left of $x_1$ we find a left parentheses and so $v_1=1$.  Next, we have that $v_2=1$ since $x_1$ and $x_2$ are adjacent and are preceded by two left parentheses.  Finally, for $v_3$ notice that just preceding $x_3$ we have a left parentheses  and so $v_3=3$.  Therefore, the left-bracket vector associated to $((x_1x_2)(x_3x_4))$ is $(1,1,3)$.

If we use left-bracket vectors, the partial order for the Tamari lattice is defined by $ (v_1,v_2, \dots, v_n) \leq (w_1,w_2,\dots, w_n)$ provided $v_i\leq w_i$ for all $1\leq i\leq n$.  In addition to the simple way the partial order is defined using left-bracket vectors, the join operation also has a nice description.
\begin{prop}[\cite{ht:pa}]\label{tamJoinProp}
Using left-bracket vector notation, the join operation in the Tamari lattice is as follows,
$$
(v_1,v_2, \dots, v_n) \vee (w_1,w_2,\dots, w_n) = (\max(v_1,w_1), \max(v_2,w_2),\dots,\max(v_n,w_n)).
$$
\end{prop}

In the next proof and the sequel we will use the notation $x\cover y$ to indicate that $y$ covers $x$.
 \begin{prop}\label{intTamariProp}
Let $T_n$ be the Tamari lattice and let I be any interval in $T_n$.  Let $\rho$ be generalized rank  as defined by equation~(\ref{genrho}) and let $m$ be the length of the longest chain in $I$. 
 If there are $k$ atoms in the interval $I$ and $\chi(I,t)$ is the characteristic polynomial with respect to $\rho$ and $m$, then
$$
\chi(I,t) = t^{m-k}(t-1)^k.
$$
\end{prop}
\begin{proof}

Partition the atoms of $I$ as $(A_1,A_2,\dots, A_k)$ where each $A_k$ has exactly one atom and use the complete transversal function $f(\bfm{t})=\vee \bfm{t}$.  With this partition  we trivially get condition (2) of Theorem~\ref{mainThmCor2}. Since we are using generalized rank, we must show that the join of any $j$ atoms is above exactly $j$ atoms in order to show condition (1). We will use the left-bracket vector representation of the elements of the Tamari lattice to verify this.

If $v$ and $w$ are left-bracket vectors and $v\cover w$, then it is easy to see that $v$ and $w$ agree in all but one position. Additionally, if we take $j$ atoms of the interval they all cover the same element in the Tamari lattice.  It follows that  each of the atoms differs from the $\hat{0}$ of the interval in one of $j$ distinct positions.   Using Proposition~\ref{tamJoinProp}, we can see that the join of $j$ atoms of the interval disagrees with the bottom element of the interval in exactly $j$ places. Let $x=a_1\vee a_2\vee\dots \vee a_j$ where the $a_i\in A(I)$.  Suppose that $b \in A(I)$ with $b\leq x$.  Then $b$ differs from the $\hat{0}$ of $I$ in exactly one place.  Moreover, since $b\leq x$, it must be one of the $j$ positions where $x$ disagrees with the $\hat{0}$ of $I$.  This implies that $b=a_i$ for some $i$. Therefore, the join of $j$ atoms is above exactly those $j$ atoms.

 Finally, we must show that $m\ge\rho(I)$ where $m$ is the length of the longest chain in $I$ since this was required in the definition of the characteristic polynomial.   Let $x_0$ be the $\hat{0}$ element of $I$ and for each $1\leq i\leq k$ define 
$$
x_i = \bigvee_{l=1}^i a_l
$$
where $a_i$ is the unique element of $A_i$.  Since the join of $j$ atoms is above exactly those $j$ atoms,  we know that all the $x_i$'s are distinct.  It follows that $I$ contains a chain of length $k$, namely the chain $x_0<x_1<\dots<x_k$.  Since $\rho$ is generalized rank  and since we partitioned the atom set into $k$ blocks, if $m$ is the length of the largest chain in $I$ then $\rho(I)=k \leq m$. Applying Theorem~\ref{mainThmCor2} now  yields the result.
\end{proof}

Let us  discuss some  consequences of this proposition.  First, since the length of the longest chain in $T_n$ is $n\choose 2$  and this poset has $n-1$ atoms,  we get that 
$$
\chi(T_n,t) = t^{n-1\choose 2}(t-1)^{n-1}
$$
which was originally shown in~\cite{bs:mfl}.

The other consequence concerns the factorization of the $m$-Tamari lattice.  Fix an $m$ and $n$,  noting that $m$ here is not being used as it was earlier in the paper. Following the definitions given in~\cite{bfp:nimt}  an \emph{m-ballot path of size $n$} is a path in the first quadrant of $\mathbb{R}^2$ from $(0,0)$ to $(mn,n)$  using unit steps north and east  which never goes below the line $x=my$.    Suppose that $P$ is an $m$-ballot path with an east step $E$ immediately followed by a north step $N$.  Another path $Q$ covers $P$ if $Q$ is obtained from $P$ by switching $E$ and $S$ where $S$ is the shortest factor of $P$ which starts at $N$ and is an $m$-ballot path. The set of $m$-ballot paths with  this covering relation defines the $m$-Tamari lattice. 

In~\cite[Proposition 4]{bfp:nimt},  it was shown that the $m$-Tamari lattices are isomorphic to intervals in the Tamari lattice.  Therefore,  we see that the characteristic polynomials of the  $m$-Tamari lattices also have a nice factorization.

Since we verified the assumptions of Theorem~\ref{mainThmCor2} in the proof of Proposition~\ref{intTamariProp}, we can also give a characterization of the M\"obius function of the intervals of the Tamari lattice.   We explain this characterization for the full Tamari lattice, but there is a similar formula for the intervals. Write a left-bracket vector $v=(v_1,v_2,\dots, v_n)$   in multiplicity notation $1^{k_1}2^{k_2}\cdots n^{k_n}$ where $k_i$ is the number of times that $i$ appears  as an entry in $v$.

In the proof of the following proposition we will make use of an equivalent definition of left-bracket vectors. As explained in~\cite{bs:mfl} a vector, $v=(v_1,v_2,\dots, v_n)$,  consisting of positive integers is a left-bracket vector  if and only if  the following hold.
\begin{enumerate}
\item For all $i$, $1\leq v_i\leq i.$
\item Letting $S_i=\{v_i, v_i+1\dots, i\}$, for any $S_i$ and $S_j$ either $S_i\cap S_j =\emptyset$ or one set is contained in the other.
\end{enumerate}  

With this equivalent definition, we can now state and prove a result about the M\"obius function of the Tamari lattice. Note the similarity of the M\"obius function of the Tamari lattice and the M\"obius function of the divisor lattice in the following proposition.

\begin{prop}\label{tamMobProp}
Let $T_n$ be the Tamari lattice.  If $v=1^{k_1}2^{k_2}\cdots n^{k_n}$ is written in multiplicity notation, then
$$
\mu(1^{k_1}2^{k_2}3^{k_3}\cdots n^{k_n})=
\begin{cases}
(-1)^{k_2+k_3+\dots+k_n}  & \mbox{if } 2^{k_2}3^{k_3}\cdots n^{k_n} \mbox{ is square free, }  \\0 & \mbox{otherwise,} 
\end{cases}
$$
 where square free means that $k_2,k_3,\dots, k_n\le 1$.
\end{prop}
 \begin{proof}
Let $v=(v_1,v_2,\dots, v_n)$  be a left-bracket vector such that written in multiplicity notation $ 2^{k_2}3^{k_3}\cdots n^{k_n}$ is square free.  We claim that in this case $v_j=1$ or $v_j=j$ for all $j$.  Suppose that this was not the case and let $v_j$ be such that $v_j\neq 1,j$. This implies that $1< v_j<j$. Additionally, $1\leq v_{v_j}<v_j$ where the last inequality is strict since $v$ is square free.  So $S_{v_j}=\{v_{v_j},v_{v_j}+1,\dots, v_j\}$ and $S_j=\{v_j, v_j+1,\dots,j\}$ with $|S_{v_j}|,|S_j|\ge2$.  Thus $S_{v_j}\cap S_j=\{v_j\}\neq\emptyset$ but neither set contains the other, which gives the desired contradiction.

In the proof of  Proposition~\ref{intTamariProp}, we showed that  the conditions of Theorem~\ref{mainThmCor2} hold if we use the partition of the atoms of $T_n$ into singleton blocks.     
Written as left-bracket vectors, the atoms of $T_n$ are of the form $(1,1,\dots, i,\dots,1)$ where $i$ is in the $i^{th}$ position.  From conclusion (b) of Theorem~\ref{mainThmCor2}, we know that the M\"obius value of an element $v\in T_n$ is $(-1)^{\rho(v)}|\mathcal{T}_v^a|$.    We showed that the join of $j$ atoms is above exactly those $j$ atoms and so each element has at most one atomic transversal.  Using Proposition ~\ref{tamJoinProp} and the previous paragraph we see that if the numbers $2,3,\dots, n$ appear at most once in the left-bracket vector of $v$, then $v$ has an atomic transversal.  And, using the same proposition, if $v$ has an atomic transversal containing $j$ atoms then it is square free.  So in this case $\mu(v)=(-1)^j$.  The rest of the elements of $T_n$ have no atomic transversal and so $\mu(v)=0$ completing the proof of the proposition.
\end{proof}

Having shown what Theorem~\ref{mainThmCor2} can say about the Tamari lattices, we now turn our attention to seeing how it implies a theorem of~\cite{bs:mfl}.  To explain both the theorem as well as how to prove it, we begin by defining the notion of a partition of an atom set being induced by a multichain.

Let $P$ be a poset with $\hat{1}$ and let $C: \hat{0}=x_0\leq x_1\leq x_2 \leq \dots \leq  x_n=\hat{1}$ be a $\hat{0}$--$\hat{1}$ multichain of $P$.  We say $(A_1,A_2,\dots, A_n)$ is \emph{induced} by $C$ if for all $i$,
$$
A_i = \{a \in A(P) \mid \hspace{ 2 pt} a\leq x_i \mbox{ and } a \nleq x_{i-1}\}.
$$
Although a partition induced by a multichain exists for any poset  with a maximum element, for the next definition we will need to have a lattice.

\begin{defi}[\cite{hs:fcp}]
Let $L$ be a lattice and let $C: \hat{0}=x_0\leq x_1\leq x_2 \leq \dots \leq  x_n=\hat{1}$ be a $\hat{0}$--$\hat{1}$ multichain. For atomic $x\in L$, $x$ neither $\hat{0}$ nor an atom, let $i$ be the index such that $x\leq x_{i}$ but $x\not\leq x_{i-1}$. We say that $C$ satisfies the \emph{meet condition} if, for each such $x$, we have $x\wedge x_{i-1} \neq \hat{0}$. 
\end{defi}

It was shown in~\cite{hs:fcp}  that if a partition is induced by a multichain, then assumption (2) of Theorem~\ref{mainThmCor2} holds if and only if the multichain satisfies the meet condition.  We will call a multichain, $C: \hat{0}=x_0\leq x_1\leq x_2 \leq \dots \leq  x_n=\hat{1}$,  \emph{saturated} if for every inequality $x_{i-1}\leq x_i$ either $x_{i-1}=x_i$ or $x_{i-1}\cover x_i$. 

Recall that   an element $x$ in a lattice $L$ is called \emph{left-modular} if for all $y,z\in L$ with $y\leq z$  we have the following equality
$$
y\vee (x\wedge z)= (y\vee x)\wedge z.
$$
We call a multichain \emph{left-modular} if every element of the multichain is left-modular.  

In~\cite{hs:fcp} it was shown that saturated $\hat{0}$--$\hat{1}$ left-modular chains satisfy the meet condition.  If $C$ is a chain which satisfies the meet condition and $M$ is a multichain formed by using all the elements of $C$  at least once, then it is not hard to see that $M$ must also satisfy the meet condition.  It follows that $\hat{0}$--$\hat{1}$ saturated left-modular  multichains satisfy the meet condition.

The authors in~\cite{hs:fcp} used the fact that partitions induced by  saturated $\hat{0}$--$\hat{1}$ left-modular chains imply  assumption (2) of Theorem~\ref{mainThmCor2} to prove Stanley's Supersolvability Theorem~\cite{s:sl}. We will use this fact to prove Blass and Sagan's result about LL lattices~\cite{bs:mfl} which is a generalization of the supersolvability result.

In order to explain this result, we need to define the level condition.   Let  $(A_1,A_2,\dots,A_n)$ be induced by $C: \hat{0}=x_0\leq x_1\leq x_2 \leq \dots \leq  x_n=\hat{1}$ .    This multichain also induces a partial ordering on the atoms denoted by $\unlhd$.  It is defined by saying  $a\lhd b$ if $a\in A_i$ and $b\in A_j$ with $i<j$.  
We say that a lattice $L$ with chain $C$ satisfies the  \emph{level condition} if
$$
a \lhd b_1 \lhd b_2 \lhd \dots \lhd b_k 
$$
 implies that
$$
a\not\leq \bigvee_{i=1}^k b_i.
$$
 
  The lattice $L$  is called an \emph{LL lattice} if it contains a left-modular multichain $C$ and $L$ together with $C$ satisfy the level condition.  We are now in a position to state Blass and Sagan's result.

\begin{thm}[\cite{bs:mfl}]
Let $L$ be a lattice and let $(A_1, A_2, \dots, A_n)$ be induced by a left-modular saturated multichain such that $L$ is an LL lattice.  Let $\rho$ be generalized rank and let $m$ be the length of the longest $\hat{0}$--$\hat{1}$ chain.  Then
$$
\chi(L,t) = t^{m-n}\prod_{i=1}^n (t-|A_i|).
$$
\end{thm}

Before we prove the theorem, let us note that it is possible to have $n>m$ in which case the exponent on the outside of the factorization will be negative.  This is possible since we are using multichains and so  repeating elements in the chain will give rise to as many empty blocks in the  partition of the atom set as we wish. However,  for  each such block, we get a corresponding factor $(t-0)$.  Thus $\chi(L,t)$ is still a polynomial since the negative power of $t$ on the outside of the product will be canceled by the positive powers of $t$ on the inside of the product.

\begin{proof}
We wish to use Theorem~\ref{mainThmCor2}. First, note that since we are using generalized rank  we have that $\rho(P)$ is at most the number of nonempty blocks in the partition.  Since our partition is induced by a multichain and since $m$ is the length of the largest chain in the lattice, we have that $\rho(P)\leq m$.   

Define the complete transversal function to be $f(\bfm{t})=\vee \bfm{t}$.  Although it is not worded in the same way, the authors in~\cite[Theorem  6.3 and Lemma 6.4] {bs:mfl}  proved assumption (1) of  Theorem~\ref{mainThmCor2} holds.  Finally, as noted before, it was shown in~\cite{hs:fcp}
 that saturated left-modular multichains satisfy the meet condition and so satisfy assumption (2) of  Theorem~\ref{mainThmCor2}.  
\end{proof}

The theorems presented so far have provided conditions which imply factorization.  We would like to finish this section with a  theorem where we provide a condition which is equivalent to factorization.

\begin{thm} \label{bigThmIff}
Let $P$ be a poset and let  $\rho: P \rightarrow \mathbb{N}$ with $m\in \mathbb{N}$ such that $\rho(P)\leq m$. Let $\chi (P,t)$ be the characteristic polynomial with respect to $\rho$ and $m$. Let  $(A_1, A_2, \dots, A_n)$ be an ordered partition of $A(P)$ and let  $f: \prod_{i=1}^nRT_{\hat{U}(A_i)}\rightarrow P$ be a complete transversal function. Finally,  define
$$
T=\{x\in P\setminus \hat{0} :\  |A_i\cap A_x|\neq 1 \mbox{ for all } i\}.
$$ 
Suppose that the following hold.
\begin{enumerate}
\item If $\bfm{t}\in \mathcal{T}_x^a$ then $|\supp(\bfm{t})|=\rho(x)$.
\item If $x, y \in P$ and $x<y$, then $\rho(x)<\rho(y)$.
\item For all minimal elements $x  ,y\in T$, the cardinality of the sets
$$
\{i :\ |A_i\cap A_x| \neq 0\} \mbox{ and } \{i :\ |A_i\cap A_y| \neq 0\}
$$
have the same parity.
\end{enumerate}

Under these conditions, 
$$
\chi(P,t)= t^{m-n}\prod_{i=1}^n (t- |A_i|)
$$
if and only if  for  every nonzero $x\in P$  there is an index $i$ such that $|A_i\cap A_x|=1$.
\end{thm}
 
This theorem is a generalization of Theorem 17 shown in~\cite{hs:fcp}. The two proofs are quite similar so we only provide a sketch below.

\begin{proof}[Sketch of proof]
First, note that the backwards  direction is Theorem~\ref{mainThmCor2}.  For the forward direction, we will prove the contrapositive.   Note that the assumption in this direction implies that $T\neq\emptyset$.   Let $k$ be the smallest value of $\rho$ applied to the elements of $T$.  We show that the coefficient of $t^{m-k}$ in $\chi(P,t)$ and in $t^{m-n}\prod_{i=1}^n (t- |A_i|)$ are different.

Define $R =  \left( \prod_{i=1}^n RT_{\hat{U}(A_i)}\right)/\ker f$.  We claim that $R$ is a homogeneous quotient and that $P\cong R$. Since $f$ is a complete transversal function, Lemma~\ref{completeTreeLem} part (c) implies that assumption (1) of Theorem~\ref{mainThm} is satisfied.  Note that the proof of part (a) of Theorem~\ref{mainThm} only requires assumption (1).  Therefore, $R$ is homogeneous and $P\cong R$.  Since $P\cong R$, it is enough to show that the coefficient of $t^{m-k}$ in $\chi( R,t)$ and in $t^{m-n}\prod_{i=1}^n (t- |A_i|)$ are not the same.

Let $x_1,x_2, \dots, x_l$ be the set of elements of $T$ with $\rho(x_i)=k$ for all $i$ and let $S=\{\mathcal{T}_{x_1},\mathcal{T}_{x_2},\dots, \mathcal{T}_{x_l}\}$ be the corresponding equivalence classes. Moreover, define $Q$ to be the poset obtained from $R$ by removing all elements of $R$ with $\rho$ value larger than $k$. Using assumption (2), we can see that the M\"obius  value of elements with $\rho$ at most $k$ in $R$ and $Q$ are the same. In $Q$ all elements with $\rho$ value $k$ are maximal. By assumption (2) and the assumption on $k$ any element of $Q$ which is not maximal cannot be in the set $T$.  Then Lemma~\ref{completeTreeLem} part (e)  implies that every non-maximal element satisfies the summation condition~\ree{sumCondEq}.  Thus, we can apply Lemma 14 in~\cite{hs:fcp} to conclude that
$$
\mu(\mathcal{T}_{x_i}) = \sum_{\bfm{t} \in \mathcal{T}_{x_i}} \mu(\bfm{t}) -\sum_{\bfm{s} \in L(\mathcal{T}_{x_i})} \mu(\bfm{s}).
$$
Let 
$$
c_i=\sum_{\bfm{s} \in L(\mathcal{T}_{x_i})} \mu(\bfm{s}).
$$
We claim that all the $c_i$'s are either 0 or have the same sign. By equation~\ree{sumOfMob}, if $c_i\neq 0$, then the sign of $c_i$ is $(-1)^{k_i}$ where $k_i$ is the number of blocks with atoms below $x_i$. 
 By assumption (3), for the $c_i$'s which are not equal to 0, the corresponding $k_i$'s have the same parity.  Therefore, the signs of the nonzero $c_i$'s are the same. Since there is at least one element of $T\neq\emptyset$ in $Q$, there is at least one $c_i\neq 0$.

Using the same argument as in the proof of Theorem 17 of~\cite{hs:fcp}, we get the coefficient of $t^{m-k}$ in $\chi(R,t)$ is 
$$
\sum_{|\supp \bfm{t}|=k} \mu(\bfm{t}) - \sum_{i=1}^l c_i
$$
in which the first sum  ranges over atomic transversals.  Since there is at least one $c_i$ which is nonzero and all the nonzero $c_i$'s have the same sign, we see that this coefficient is not the same as
$$
\sum_{|\supp\bfm{t}|=k} \mu(\bfm{t}).
$$
However,   the previous expression  is the coefficient of $t^{m-k}$ in $t^{m-n}\prod_{i=1}^n (t- |A_i|)$.  It follows that
$$
\chi(P,t) \neq  t^{m-n}\prod_{i=1}^n (t- |A_i|)
$$
which is what we wished to show.
\end{proof}

\section{Classic Results About the  M\"obius Function}\label{mobFunSec}

In this section, we will give a new method to prove an array of classic results about the M\"obius function. The idea of the method is to use induction on the size of the poset.  In order to do this, we will collapse a coatom and the $\hat{1}$ of the poset.

 We begin with a lemma that explains the simple nature of the values of $\mu$ for the original poset and the poset obtained by collapsing a coatom and $\hat{1}$. In the lemma and throughout the rest of the section, we will use $[x]$ to denote the equivalence class which contains $x$.

\begin{lem}\label{coatomAndOneCollapse}
Let $P$ be a poset with a $\hat{0}$ and $\hat{1}$ and at least 3 elements.  Suppose $c$ is a coatom and let $\sim$ be the equivalence relation identifying $c$ and $\hat{1}$.   Then $P/\sim$ is homogeneous and 
$$
\mu([\hat{1}])=\mu(c)+\mu(\hat{1}).
$$
Moreover, if $P$ is a lattice, then $P/\sim$ is a lattice with $[x]\vee[y]=[x\vee y]$ for all $x,y\in P$ and $[x]\wedge [y]=[x\wedge y]$ provided $[x],[y]\neq [\hat{1}]$.
\end{lem}
\begin{proof}
First, let us show that $P/\sim$ is homogeneous.  Since there are at least 3 elements and we are collapsing a coatom and $\hat{1}$, we have that $\hat{0}$ is in its own equivalence class.  Now suppose that $[x]<[y]$.  It follows that  $[x]\neq \{c,\hat{1}\}$ since $[x]<[y]$ and  $[c]=[\hat{1}]$ is the $\hat{1}$ of the quotient.  Therefore, $[x]=\{x\}$ and so it is obvious that $P/\sim$ is a homogeneous quotient.

To show that $\mu([\hat{1}])=\mu(c)+\mu(\hat{1})$ note that since every element of $P$ is below $\hat{1}$ and every other equivalence class has only one element, we get  
$$
\sum_{y\in L([x])}\mu(y)= \sum_{y\leq x}\mu(y) = 0
$$
for all nonzero $x\neq c$.
By Lemma~\ref{sumLem} this implies that 
$$
\mu([\hat{1}])=\mu(c)+\mu(\hat{1})
$$
which is what we wished to prove.

Now suppose that $P$ is a lattice.   It is not hard to see that $(P/\sim)\hspace{4 pt} \cong (P\setminus \{c\})$. Therefore, if $x\vee y\neq c$, we immediately get that 
$[x]\vee[y]$ exists and $[x]\vee[y]=[x\vee y]$.  If $x\vee y=c$, then $\hat{1}$ is the  only  element in $P\setminus \{c\}$ which is an upper bound for both $x$ and $y$.  It follows that $[x]\vee[y]=[\hat{1}]=[c]=[x\vee y]$.  Since $P\setminus \{c\}$ clearly has a $\hat{0}$, we conclude $P/\sim$ is a lattice.    Finally, if $[x],[y]\neq [\hat{1}]$ then $[x]=\{x\}$ and $[y]=\{y\}$ and so $[x]\wedge [y]=[x\wedge y]$.
\end{proof}

Let us now use Lemma~\ref{coatomAndOneCollapse} to prove some classic results.

\begin{cor}[Hall's Theorem~\cite{h:efg}]
Let $P$ be a finite poset, then 
$$
\mu(x,y)=\sum_{i\geq 0} (-1)^i c_i
$$
where $c_i$ is the number of chains of length $i$ which start at $x$ and terminate at $y$.
\end{cor}
\begin{proof} Without loss of generality we may assume that $x=\hat{0}$ and $y=\hat{1}$ since all chains which start at $x$ and terminate at $y$ are in the interval $[x,y]$. We prove the theorem by inducting on $|P|$.  If $|P|=1$ or $|P|=2$ then the result is obvious.

Now suppose that $|P|>2$.  Let $P/\sim$ be obtained by identifying a coatom $c$ and $\hat{1}$.   Consider the sum
$$
\sum_{i\geq 0} (-1)^i c_i
$$
where $c_i$ is the number of $\hat{0}$--$\hat{1}$ chains of length $i$  in $P$.  Let $a_i$ be the number chains of length $i$ which do not contain $c$ and let $b_i$ be the number chains of length $i$ containing $c$.  Then
$$
\sum_{i\geq 0} (-1)^i c_i = \sum_{i\geq 0} (-1)^i a_i+\sum_{i\geq 0} (-1)^i b_i.
$$
There exists a bijection between $\hat{0}$--$\hat{1}$ chains in $P$ not containing $c$ and $[\hat{0}]$--$[\hat{1}]$ chains in $P/\sim$ which preserves length.  Moreover,  there is a bijection between $\hat{0}$--$\hat{1}$ chains in $P$ containing $c$ and $\hat{0}$--$c$ chains in $[0,c]$.  Note that in this bijection, the chains decrease by one in length.

Since $|P/\sim|<|P|$, using induction we get that 
$$
\mu([\hat{1}]) =  \sum_{i\geq 0} (-1)^i a_i.
$$
Similarly since $|[0,c]|<|P|$ we get that 
$$
\mu(c) = -\sum_{i\geq 0} (-1)^i b_i
$$
where we have multiplied the sum  by $-1$ since the chains have decreased by one in length.

By Lemma~\ref{coatomAndOneCollapse}, we have that
$$
\mu([\hat{1}])=\mu(c)+\mu(\hat{1})
$$
or equivalently
$$
\mu(\hat{1}) = \mu([\hat{1}])-\mu(c).
$$
Therefore,
$$
\mu(\hat{1}) =  \sum_{i\geq 0} (-1)^i a_i+\sum_{i\geq 0} (-1)^i b_i=\sum_{i\geq 0} (-1)^i c_i 
$$
which is what we wished to prove.
\end{proof}

Next, we prove a theorem of Weisner.

\begin{cor}[Weisner's Theorem \cite{w:atifs}]
Let $L$ be a lattice and let $ \hat{0}\neq a\in L$.  If $|L|\geq2$, then
$$
\mu(\hat{1}) =-\sum_{x\neq \hat{1}, x\vee a =\hat{1}}\mu(x).
$$
\end{cor}
\begin{proof} Let us note that if $a=\hat{1}$, then the result is just restating the definition of $\mu$,  so we assume that $a\neq\hat{1}$ for the rest of the proof. We prove the result by induction. We have already covered the case $|L|=2$, since then $a$ must be $\hat{1}$.

Now suppose that $|L|>2$.  Let $c$ be a coatom such that $a\leq c$.  Consider, the lattice $L/\sim$ obtained by identifying $c$ and $\hat{1}$.  Since $|L/\sim|<|L|$, we get that 
$$
\mu([\hat{1}]) = -\sum_{[x]\neq [\hat{1}], [x]\vee[a] =[\hat{1}]}\mu([x]).
$$
Using the facts that $[\hat{1}]=\{c,1\}$,  $[x]\vee[a]=[x\vee a]$, and  $\mu([x])=\mu(x)$ for $[x]\neq [\hat{1}]$, we obtain,
$$
\mu([\hat{1}]) = -\sum_{x\neq c,\hat{1}, x\vee a =c,\hat{1}}\mu(x).
$$
Since joins are unique, we can break the sum into two parts as,
$$
\mu([\hat{1}]) = -\sum_{x\neq c,\hat{1}, x\vee a =c}\mu(x)  -\sum_{x\neq c,\hat{1}, x\vee a =\hat{1}}\mu(x).
$$
If $x\vee a= c$, then it is clear that $x \in [0,c]$.  Moreover, since $a\leq c$, it is clear that $c\vee a \neq \hat{1}$.  Thus, we can  remove the $x\neq \hat{1}$ condition in the first sum and remove the $x\neq c$ condition in the second.  This gives,
$$
\mu([\hat{1}]) = -\sum_{x\neq c, x\vee a =c}\mu(x)  -\sum_{x\neq  \hat{1}, x\vee a =\hat{1}}\mu(x).
$$
Now the first sum is only over $[0,c]$ and $|[0,c]|<|L|$ so by induction,
$$
\mu([\hat{1}]) =\mu(c) -\sum_{x\neq  \hat{1}, x\vee a =\hat{1}}\mu(x).
$$
Using the fact that $\mu([\hat{1}])=\mu(\hat{1})+\mu(c)$, we  immediately obtain the result.
\end{proof}

Our next corollary will make use of crosscuts.  We remind the reader of the definition here.
\begin{defi}
Let $L$ be a lattice.  A \emph{crosscut} of $L$ is a set $C$ 
with the following properties:
\begin{enumerate}
\item $\hat{0}, \hat{1}\notin C$.
\item $C$ is an antichain.
\item Every maximal $\hat{0}$--$\hat{1}$ chain intersects $C$.
\end{enumerate}
\end{defi}

\begin{thm}[Rota's Crosscut Theorem~\cite{r:ofct}]
Let $L$ be a lattice and let $C$ be a crosscut.  Then
$$
\mu(\hat{1}) = \sum_{\vee B=\hat{1}, \wedge B=\hat{0}} (-1)^{|B|}
$$
where the sum ranges over all $B\subseteq C$ such that $\vee B=\hat{1}$ and $\wedge B=\hat{0}$.
\end{thm}

\begin{proof}
 We first consider the special case when every coatom is also an atom. In this case, the crosscut must be the atom set.  Moreover, a subset of the crosscut has meet $\hat{0}$ and join $\hat{1}$ if and only if it has at least two elements.  Therefore, if $L$ has $n$ atoms   we obtain the following
$$
 \sum_{\vee B=\hat{1}, \wedge B=\hat{0}} (-1)^{|B|} = \sum_{|B|\geq 2} (-1)^{|B|} = \sum_{k=2}^n (-1)^k {n\choose k}=n-1.
$$
This agrees with the value of $\mu(\hat{1})$ when $L$ has $n$ atoms and every coatom is an atom. Thus, the result holds in this special case.

Recall that if $L^*$ is the dual lattice of $L$, then $\mu_L(\hat{1}) = \mu_{L^*}(\hat{1})$.  Moreover, in $L^*$, joins and meets reverse roles.    Therefore, if we have a crosscut consisting of only  coatoms, then we can consider the dual lattice.  As a result, we may now assume that there is always at least one coatom in the lattice which is not in the crosscut.  With this in mind we proceed by induction on  $|L|$.  

If $|L|\leq 3$, then it must be that $|L|=3$ since smaller lattices do not have crosscuts.  We have already done the case when $|L|=3$.  Suppose that $|L|>3$ and let $c$ be a coatom that is not in the crosscut.  Consider the lattice $L/\sim$ where we collapse $c$ and $\hat{1}$.  Since $c$ was not in the crosscut we still have the same crosscut.
 By induction, we know that
$$
\mu([\hat{1}])= \sum_{\vee B=[\hat{1}],\wedge B =[\hat{0}]} (-1)^{|B|}.
$$
  Lemma~\ref{coatomAndOneCollapse} implies that   $\vee B = [\hat{1}]$ in $L/\sim$  if and only if $\vee B=c$ or $\vee B =\hat{1}$ in $L$.  Additionally, since $C$ does not contain $c$ nor $\hat{1}$, Lemma~\ref{coatomAndOneCollapse}  also implies that $\wedge B=[\hat{0}]$ in $L/\sim$ if and only if $\wedge B =\hat{0}$ in $L$. Therefore, we can break the previous sum  as follows
$$
\mu([\hat{1}])= \sum_{\vee B=c,\wedge B=\hat{0}} (-1)^{|B|} + \sum_{\vee B=\hat{1},\wedge B=\hat{0}} (-1)^{|B|}.
$$

Note that  if $\vee B=c$, then $B$ must  only have elements in $[\hat{0},c]$.  Thus the first sum in the previous equation is over $B$  contained in $[\hat{0},c]\cap C$ such that $\vee B=c$ and $\wedge B =\hat{0}$.  Since $|[\hat{0},c]|<|L|$, induction implies that
$$
\mu([\hat{1}])= \mu(c) + \sum_{\vee B=\hat{1},\wedge B=\hat{0}} (-1)^{|B|}.
$$
Subtracting $\mu(c)$ from both sides and applying Lemma~\ref{coatomAndOneCollapse} we see that
$$
\mu(\hat{1}) =  \sum_{\vee B=\hat{1},\wedge B=\hat{0}} (-1)^{|B|}
$$
which completes the proof.
\end{proof}

\section{Future Work}

We saw in Proposition~\ref{tamMobProp} that the M\"obius function for the Tamari lattice has a  description which is similar to that of the divisor lattice.  In the future, we hope that we can show this fact by exhibiting a quotient of the divisor lattice which preserves the M\"obius function.  More generally, it would be nice to find other examples of unranked posets whose M\"obius function behaves like that of a ranked poset.

It is easy to see that given a partition of the atom set of a poset and an element $x$, the elements of the set $\mathcal{T}_x $ can be viewed as the facets of an abstract simplicial complex $\Delta(\mathcal{T}_x )$. This complex encodes some useful information. For example, if $\bfm{t}\in \mathcal{T}_x^a$ then as an element of the product of rooted trees  $\mu(\bfm{t})= (-1)^{|\supp\bfm{t}|}$.  Therefore  
$$
\sum_{\bfm{t} \in L(\mathcal{T}_x)} \mu(\bfm{t}) = \sum_{\bfm{t} \in L(\mathcal{T}^a_x)} (-1)^{|\supp\bfm{t}|}.
$$
Since the dimension of $\bfm{t}$ in $\Delta(\mathcal{T}_x)$ is $|\supp \bfm{t}| -1$ the previous equation can be rewritten as
$$
\sum_{\bfm{t} \in L(\mathcal{T}_x)} \mu(\bfm{t}) = \sum_{\bfm{t} \in L(\mathcal{T}^a_x)} (-1)^{\dim \bfm{t}+1}.
$$
This is equivalent to
$$
\sum_{\bfm{t} \in L(\mathcal{T}_x)} \mu(\bfm{t}) = - \sum_{\bfm{t} \in L(\mathcal{T}^a_x)} (-1)^{\dim \bfm{t}}.
$$
One can see that the right-hand side of the previous equation is the negative of the reduced Euler characteristic of  $\Delta(\mathcal{T}^a_x)$.  It follows that the summation condition~\ree{sumCondEq} is satisfied if and only if for each $x$ the reduced Euler characteristic of $\Delta(\mathcal{T}_x^a)$ is $-\delta_{\hat{0},x}$.  
We also note that if we restrict to the set of atomic transversals for $x$, $\Delta(\mathcal{T}_x^a)$ is pure of dimension $\rho(x)-1$ if and only if condition (2) of Theorem~\ref{mainThm} holds. We are interested in investigating what else this complex can tell us about the poset.

 \emph{Acknowledgement.}
The author would like to thank Bruce Sagan for his helpful discussions as well as his help preparing the manuscript.

\end{document}